\documentclass[10pt, reqno]{amsart}
\usepackage{amsmath, amsthm, amscd, amsfonts, amssymb, graphicx, color}
\usepackage[bookmarksnumbered, colorlinks, plainpages]{hyperref}
\hypersetup{colorlinks=true,linkcolor=red, anchorcolor=green, citecolor=cyan, urlcolor=red, filecolor=magenta, pdftoolbar=true}
\usepackage{mathrsfs}

\newtheorem{theorem}{Theorem}[section]
\newtheorem{lemma}[theorem]{Lemma}

\newtheorem{corollary}[theorem]{Corollary}
\theoremstyle{definition}

\newtheorem{example}[theorem]{Example}

\theoremstyle{remark}
\newtheorem{remark}[theorem]{Remark}
\numberwithin{equation}{section}

\begin{document}
\setcounter{page}{1}

\title[K-theory  of rational quadratic forms]{K-theory  of rational quadratic forms}

\author[I.Nikolaev]
{Igor Nikolaev$^1$}

\address{$^{1}$ Department of Mathematics and Computer Science, St.~John's University, 8000 Utopia Parkway,  
New York,  NY 11439, United States.}
\email{\textcolor[rgb]{0.00,0.00,0.84}{igor.v.nikolaev@gmail.com}}





\begin{abstract}
We compute the genus of a rational quadratic form in terms of 
the K-theory of  a  $C^*$-algebra  attached to the adelic orthogonal group of the form.  
As a corollary,  one gets  a higher composition law for  the rational quadratic forms.  
As an illustration, we consider  the Gauss composition  of the 
binary quadratic forms. 
\end{abstract}

\maketitle

\section{Introduction}
The binary quadratic forms $q(u,v)=\{au^2+buv+cv^2 ~|~ a,b,c\in\mathbf{Z}\}$
were studied by C.-F.~Gauss.  Two forms $q$ and $q'$  are said to be equivalent,
if a substitution $\{u=\alpha u'+\beta v', ~v=\gamma u' +\delta v' ~|~\alpha,\beta,\gamma,\delta\in\mathbf{Z},
~\alpha\delta-\beta\gamma=1\}$ transforms $q$ into $q'$. It is easy to see that the discriminant $\Delta=b^2-4ac$
of the form $q(u,v)$ is an invariant of the substitution and, therefore,  the equivalent forms  have the same discriminant. 
But the converse  is false in general.  Gauss showed that there exists a finite number of the pairwise non-equivalent 
binary quadratic forms having the same discriminant.   Moreover,  the equivalence classes make  an abelian 
group $\mathcal{G}$ under a composition defined on these  forms [Cassels 1978] \cite[Chapter 14]{C}.  
The cardinality $g=|\mathcal{G}|$ of such a group is called the genus of $q(u,v)$. 
An  extension of  the Gauss composition to the general quadratic forms 
$q(x):=q(x_1,\dots,x_n)=\{\sum_{i=1}^n \sum_{j=1}^n a_{ij} x_i x_j ~|~ a_{ij}\in\mathbf{Z}, ~n\ge 1\}$
is a difficult and important problem [Bhargava 2004] \cite{Bha1}, \cite{Bha2}, \cite{Bha3} and [Bhargava 2008] \cite{Bha4},
see  the  survey   [Bhargava 2006] \cite{Bha5}.

Let $G(K)$ be a non-compact  reductive algebraic group defined over a number field $K$. 
Denote by  $\mathbb{A}$ the ring of adeles of $K$ and by $\mathbb{A}_{\infty}$ 
a subring of  the integer adeles.
It is well known that  $G(K)$ is a discrete subgroup of $G(\mathbb{A})$ and 
the double cosets  $G(\mathbb{A}_{\infty})\backslash G(\mathbb{A})/G(K)$ is a finite set.
The  set $G(\mathbb{A}_{\infty})\backslash G(\mathbb{A})/G(K)$ is an important arithmetic invariant
of the group $G(K)$. 
For instance, if $G\cong O(q)$ is the orthogonal group of a quadratic form,  
then  $|G(\mathbb{A}_{\infty})\backslash G(\mathbb{A})/G(K)|=g$
 [Platonov \& Rapinchuk 1994] \cite[Chapter 8]{PR}.

Consider a  Banach algebra $L^1(G(\mathbb{A})/G(K))$  of 
the  integrable complex-valued  functions on the homogeneous space $G(\mathbb{A})/G(K)$ 
endowed  with the operator norm.   Recall that the addition of functions 
$f_1, f_2\in L^1(G(\mathbb{A})/G(K))$ is defined pointwise  and 
 multiplication is given by the convolution product:
\begin{equation}\label{eq1.1}
(f_1\ast f_2)(g)=\int_{G(\mathbb{A})/G(K)} f_1(gh^{-1})f_2(h)dh. 
\end{equation}  
By  $\mathscr{A}$ we understand  an enveloping $C^*$-algebra   of  the algebra $L^1(G(\mathbb{A})/G(K))$;
we refer the reader to  [Dixmier 1977]  \cite[Section 13.9]{D}  for details of this construction.

The aim of our note is a map from  the set $G(\mathbb{A}_{\infty})\backslash G(\mathbb{A})/G(K)$ to 
 the K-theory  of the $C^*$-algebra $\mathscr{A}$.   Since the K-groups are 
 abelian,  one gets the structure of an abelian group  on  the  set  $G(\mathbb{A}_{\infty})\backslash G(\mathbb{A})/G(K)$. 
 As an application, we consider the case $G\cong O(q)$, where $O(q)$ is the orthogonal group of the rational 
quadratic form $q(x)$. In this case,  one gets a higher composition law for the quadratic forms and a formula 
for their genera.  To formulate our results, we shall need the following definitions.

It is known that $\mathscr{A}$ is a stationary AF-algebra of rank $n=rank~G(K)$ \cite[Lemma 3.1]{Nik1}. 
Such an algebra is  given  by an  $n\times n$ integer matrix $A$ with  the determinant $\det~A=1$  [Effros 1981] \cite[Chapter 6]{E}.
Denote by $\mathscr{A}\rtimes_{\sigma}\mathbf{Z}$ the crossed product  of the $C^*$-algebra $\mathscr{A}$ 
by the shift automorphism $\sigma$ of $\mathscr{A}$  [Blackadar 1986] \cite[Exercise 10.11.9]{B}.
By $K_0$ we understand the zero $K$-group of an algebra.
It is known  that $K_0(\mathscr{A}\rtimes_{\sigma}\mathbf{Z})\cong \mathbf{Z}^n/(I-A)\mathbf{Z}^n$ and 
$|K_0(\mathscr{A}\rtimes_{\sigma}\mathbf{Z})|=|\det ~(I-A)|$, where $I$ is the identity matrix 
[Blackadar 1986] \cite[Theorem 10.2.1]{B}.
Our main results can be formulated as follows.
\begin{theorem}\label{thm1.1}
There exists a one-to-one map
\begin{equation}\label{eq1.2}
\phi: G(\mathbb{A}_{\infty})\backslash G(\mathbb{A})/G(K)\to K_0(\mathscr{A}\rtimes_{\sigma}\mathbf{Z}).
\end{equation}
In particular, the map (\ref{eq1.2})  defines the structure of an abelian group on the set $G(\mathbb{A}_{\infty})\backslash G(\mathbb{A})/G(K)$,  
so that $\phi$  becomes  an isomorphism of the abelian groups.
\end{theorem}
\begin{corollary}\label{cor1.2}
If $G\cong O(q)$ is the orthogonal group of a rational quadratic form $q$,  then
(\ref{eq1.2}) defines a composition law  for  the equivalence classes of 
the quadratic forms.   In particular, the genus $g$ of  the quadratic form $q$ is given by 
the formula 
\begin{equation}\label{eq1.02}
g=|\det ~(I-A)|.
\end{equation}
\end{corollary}
The article is organized as follows. In Section 2 we briefly review the AF-algebras, crossed products   
and algebraic groups over the ring of adeles.  Theorem \ref{thm1.1} and corollary \ref{cor1.2} are 
proved  in Section 3.  An illustration of corollary \ref{cor1.2} can be found in Section 4.

\section{Preliminaries}
The $C^*$-algebras are covered in [Dixmier 1977]  \cite{D}.  
For the AF-algebras we refer the reader to [Effros 1981]  \cite{E}.
The K-theory of $C^*$-algebras is the subject of  [Blackadar 1986] \cite{B}.
An excellent introduction to the algebraic groups can be found in  
[Platonov \& Rapinchuk 1994] \cite{PR}. 

\subsection{AF-algebras}
\subsubsection{$C^*$-algebras}
The  $C^*$-algebra is an algebra $\mathcal{A}$ over $\mathbf{C}$ with a norm
$a\mapsto ||a||$ and an involution $a\mapsto a^*$ such that
it is complete with respect to the norm and $||ab||\le ||a||~ ||b||$
and $||a^*a||=||a^2||$ for all $a,b\in \mathcal{A}$.
Any commutative $C^*$-algebra is  isomorphic
to the algebra $C_0(X)$ of continuous complex-valued
functions on a locally compact Hausdorff space $X$; 
otherwise, the algebra $\mathcal{A}$ represents a noncommutative  topological
space.

\subsubsection{Crossed products}
Let $\sigma: \mathcal{A}\to\mathcal{A}$ be an automorphism of the $C^*$-algebra
$\mathcal{A}$. The idea of a crossed product $C^*$-algebra $\mathcal{A}\rtimes_{\sigma} \mathbf{Z}$
by $\sigma$ is to embed $\mathcal{A}$ into a larger $C^*$-algebra in which the automorphism 
$\sigma$ becomes an inner automorphism. Namely, let $G$ be a locally compact group 
and $\sigma$ a continuous homomorphism from $G$ into $Aut ~(\mathcal{A})$. 
A covariant representation of the triple $(\mathcal{A}, G, \sigma)$ is a pair of 
representations $(\pi,\rho)$ of $\mathcal{A}$ and $G$ on a Hilbert space,
such that $\rho(g)\pi(a)\rho(g)^*=\pi(\sigma_g(a))$ for all $a\in\mathcal{A}$ and 
$g\in G$.   The covariant representation defines an algebra whose completion in 
the operator norm is called a crossed product $C^*$-algebra $\mathcal{A}\rtimes_{\sigma} G$.

\subsubsection{AF-algebras}
An  AF-algebra  (Approximately Finite $C^*$-algebra) is defined to
be the  norm closure of an ascending sequence of   finite dimensional
$C^*$-algebras $M_n$,  where  $M_n$ is the $C^*$-algebra of the $n\times n$ matrices
with entries in $\mathbf{C}$. The index $n=(n_1,\dots,n_k)$ represents
a  semi-simple matrix algebra $M_n=M_{n_1}\oplus\dots\oplus M_{n_k}$.
The ascending sequence can be written as 
\begin{equation}\label{eq2.1}
M_1\buildrel\rm\varphi_1\over\longrightarrow M_2
   \buildrel\rm\varphi_2\over\longrightarrow\dots,
\end{equation}
where $M_i$ are the finite dimensional $C^*$-algebras and
$\varphi_i$ the homomorphisms between such algebras.  
If $\varphi_i=Const$, then the AF-algebra $\mathcal{A}$ is called 
stationary.   The  shift automorphism of  a stationary AF-algebra 
corresponds to the map $i\mapsto i+1$  in (\ref{eq2.1})
[Effros 1981]  \cite[p.37]{E}.

The homomorphisms $\varphi_i$ can be arranged into  a graph as follows. 
Let  $M_i=M_{i_1}\oplus\dots\oplus M_{i_k}$ and 
$M_{i'}=M_{i_1'}\oplus\dots\oplus M_{i_k'}$ be 
the semi-simple $C^*$-algebras and $\varphi_i: M_i\to M_{i'}$ the  homomorphism. 
One has  two sets of vertices $V_{i_1},\dots, V_{i_k}$ and $V_{i_1'},\dots, V_{i_k'}$
joined by  $a_{rs}$ edges  whenever the summand $M_{i_r}$ contains $a_{rs}$
copies of the summand $M_{i_s'}$ under the embedding $\varphi_i$. 
As $i$ varies, one obtains an infinite graph called the   Bratteli diagram of the
AF-algebra.  The $A=(a_{rs})$ is called  a  matrix of the partial multiplicities;
an infinite sequence of $A_i$ defines a unique AF-algebra. If $rank ~(A_i)=n=Const$, then 
the AF-algebra is said to be of {\it rank} $n$.   If the AF-algebra $\mathcal{A}$
is stationary,  then $A_i=A$, where $A$ is an $n\times n$ matrix with the non-negative
entries.

\subsubsection{$K_0$-groups}
For a  $C^*$-algebra $\mathcal{A}$,  let $V(\mathcal{A})$
be the union of projections in the $k\times k$
matrix $C^*$-algebra with entries in $\mathcal{A}$ taken over all $k=1,2,\dots,\infty$. 
The projections $p,q\in V(\mathcal{A})$ are  equivalent,  if there exists a partial
isometry $u$ such that $p=u^*u$ and $q=uu^*$. The equivalence
class of projection $p$ is denoted by $[p]$. 
The equivalence classes of orthogonal projections can be made to
a semigroup by putting $[p]+[q]=[p+q]$. The Grothendieck
completion of this semigroup to an abelian group is called
the  $K_0$-group of the algebra $\mathcal{A}$.
The functor $\mathcal{A}\to K_0(\mathcal{A})$ maps a category of the 
$C^*$-algebras into the category of abelian groups, so that
projections in the algebra $\mathcal{A}$ correspond to a positive
cone  $K_0^+\subset K_0(\mathcal{A})$ and the unit element $1\in \mathcal{A}$
corresponds to an order unit $u\in K_0(\mathcal{A})$.

Let $\mathcal{A}$ be the stationary AF-algebra given by an $n\times n$  matrix $A$.
Denote by $\sigma: \mathcal{A}\to\mathcal{A}$ the shift automorphism
of $\mathcal{A}$. Consider a crossed product $C^*$-algebra 
$\mathcal{A}\rtimes_{\sigma} \mathbf{Z}$ defined by the automorphism $\sigma$.  
The $K_0$-group of the crossed product $\mathcal{A}\rtimes_{\sigma} \mathbf{Z}$
is given by the formula:
\begin{equation}\label{eq2.2}
K_0(\mathcal{A}\rtimes_{\sigma} \mathbf{Z})\cong {\mathbf{Z}^n\over (I-A)\mathbf{Z}^n}, \quad \hbox{where} ~I=diag ~(1,1,\dots,1).
\end{equation}

\subsection{Algebraic groups over adeles}
\subsubsection{Algebraic groups}
An algebraic group is an algebraic variety $G$ together with (i) an element 
$e\in G$, (ii) a morphism $\mu:  G\times G\to G$ given by the formula
$(x,y)\mapsto xy$ and (iii) a morphism $i: G\to G$ given by the formula
$x\mapsto x^{-1}$ with respect to which the set $G$ is a group.
If $G$ is a non-compact variety, the algebraic group $G$ is said to be 
 non-compact.   The $G$ 
is called a  $K$-group if $G$ is a variety defined over the field $K$
and if $\mu$ and $i$ are defined over $K$.  In what follows, we deal 
with the linear algebraic groups, i.e. the subgroups of the general linear
group $GL_n$.  
An algebraic group $G$ is called reductive  if the unipotent radical of $G$
is trivial. An informal equivalent definition says that $G$ is reductive if and only if
a representation of $G$ is a direct sum of the irreducible representations.

\subsubsection{Orthogonal group of a quadratic form}
Let $q(x)=\{\sum_{i=1}^n \sum_{j=1}^n a_{ij} x_i x_j ~|~ a_{ij}\in\mathbf{Z}, ~n\ge 1\}$ be a quadratic form. 
Roughly speaking, the orthogonal group is a subgroup of the $GL_n$ which preserves the form $q(x)$. 
Namely, let $A=(a_{ij})$ be a symmetric matrix attached to the quadratic form $q(x)$.  
The $O(q)=\{g\in GL_n ~|~ gAg=A\}$ is called an  orthogonal group and 
$SO(q)=\{g\in O(q) ~|~ \det~g=1\}$ is called a  special  orthogonal group
of the form $q(x)$.

\subsubsection{Ring of adeles}
The adeles is a powerful tool describing  the Artin reciprocity for   the abelian extensions of a number field $K$. 
In intrinsic terms, the  ring of adeles  $\mathbb{A}$ of a field  $K$ is a subset of the direct product $\prod K_{v}$
taken over almost all places $K_{v}$ of  $K$ endowed with the natural topology. 
The  embeddings $K\hookrightarrow K_v$ induce a  discrete diagonal  embedding $K\hookrightarrow\mathbb{A}$; 
the image of such is  a  ring of  the principal adeles  of $\mathbb{A}$. 
The ring of integral adeles $\mathbb{A}_{\infty}:=\prod \mathcal{O}_v$, where $\mathcal{O}_v$ is a localization of the ring 
$\mathcal{O}_{K}$ of the integers of the field $K$.  The Artin reciprocity says that there exists a continuous homomorphism 
$\mathbb{A}^{\times}\to Gal~(K^{ab}|K)$,  where $\mathbb{A}^{\times}$ is a group of the invertible adeles (the idele group) 
and $Gal~(K^{ab}|K)$ is the absolute Galois group of the  abelian extensions of $K$ endowed with a profinite topology. 
On the other hand, there exists a canonical isomorphism $\mathbb{A}_{\infty}^{\times}\backslash\mathbb{A}^{\times}/K^{\times}\to Cl~(K)$,
where $\mathbb{A}_{\infty}^{\times}$ ($\mathbb{A}^{\times}$ and  $K^{\times},$ resp.)  is a  group of units of  the ring  $\mathbb{A}_{\infty}$ 
($\mathbb{A}$ and  $\mathcal{O}_{K}$, resp.)  and  $Cl~(K)$ is the ideal class group of the field $K$.

\subsubsection{Algebraic groups over adeles}
The algebraic groups over the ring of adeles can be viewed as an analog of the Artin recipricity for
the non-abelian extensions of the field $K$ and, therefore, we deal with a  noncommutative 
arithmetic  [Platonov \& Rapinchuk 1994] \cite[p. 243]{PR}.  Such groups are a starting point of the 
Langlands program [Langlands 1978] \cite{Lan1}.  

Let $G$ be an algebraic group.  It is known that  the double cosets  \linebreak
$G(\mathbb{A}_{\infty})\backslash G(\mathbb{A})/G(K)$
is a finite set [Borel 1963] \cite{Bor1}. In particular, if $G\cong GL_n$ then  
$|G(\mathbb{A}_{\infty})\backslash G(\mathbb{A})/G(K)|=h_{K}$ and
if $G\cong O(q)$ then $|G(\mathbb{A}_{\infty})\backslash G(\mathbb{A})/G(K)|=g$,
where $h_{K}=|Cl~(K)|$ is the class number of the field $K$ and $g$ is the genus of 
quadratic form $q(x)$, respectively. The cardinality of the set $G(\mathbb{A}_{\infty})\backslash G(\mathbb{A})/G(K)$
is a difficult open  problem.

\section{Proofs}
\subsection{Proof of theorem \ref{thm1.1}}
We shall split the proof in a series of lemmas.
\begin{lemma}\label{lm1}
Consider the group $Ext~(\mathscr{A}\rtimes_{\sigma}\mathbf{Z})$ consisting of 
the equivalence classes of extensions
\begin{equation}\label{eq3.1}
1\to\mathbf{C}\to E\to \mathscr{A}\rtimes_{\sigma}\mathbf{Z}\to 1
\end{equation}
of the  $\mathscr{A}\rtimes_{\sigma}\mathbf{Z}$ by the complex numbers $\mathbf{C}$. 
Then  $Ext~(\mathscr{A}\rtimes_{\sigma}\mathbf{Z})\cong K_0(\mathscr{A}\rtimes_{\sigma}\mathbf{Z})$. 
 \end{lemma}
\begin{proof}
The result follows from the equivalence of the $\mathbf{Ext}$ and the $\mathbf{K}_0$-functors 
on the category of $C^*$-algebras. The special case  of the $C^*$-algebra 
$\mathscr{A}\rtimes_{\sigma}\mathbf{Z}:=O_A$ can be found  in
[Blackadar 1986] \cite[Section 16.4.5]{B}.  
\end{proof}

\begin{lemma}\label{lm2}
Denote by  $E_0\cong \left(\mathscr{A}\rtimes_{\sigma}\mathbf{Z}\right)\oplus\mathbf{C}$
the trivial extension of the crossed product $\mathscr{A}\rtimes_{\sigma}\mathbf{Z}$
corresponding to the null element of the group  $Ext~(\mathscr{A}\rtimes_{\sigma}\mathbf{Z})$.
Then $E_0$ is a non-commutative coordinate ring of the group  $G(\mathbb{A})$,
i.e. there exists a covariant functor between the respective categories, such that 
the morphisms of $G(\mathbb{A})$ correspond to  the morphisms of $E_0$. 
\end{lemma}
\begin{proof}
(i) Recall that the AF-algebra $\mathscr{A}$ is a non-commutative coordinate ring of the 
group  $G(\mathbb{A})$ \cite[Lemma 3.1]{Nik1}. 
Consider the crossed product $\mathscr{A}\rtimes_{\sigma}\mathbf{Z}$
by the shift automorphism $\sigma$ of $\mathscr{A}$. It is well known, 
that there exists a canonical embedding of the $C^*$-algebras: 
\begin{equation}\label{eq3.2}
\mathscr{A}\hookrightarrow \mathscr{A}\rtimes_{\sigma}\mathbf{Z}.
\end{equation}
Using (\ref{eq3.2}) one can extend each morphism of $\mathscr{A}$ to 
such of the crossed product  $\mathscr{A}\rtimes_{\sigma}\mathbf{Z}$.
Such an extension is well defined and unique. Thus one gets a functor 
between the category of groups $G(\mathbb{A})$ and the category of 
crossed products $\mathscr{A}\rtimes_{\sigma}\mathbf{Z}$. In other words,
the  $\mathscr{A}\rtimes_{\sigma}\mathbf{Z}$ is a non-commutative coordinate
ring of the group $G(\mathbb{A})$.  

\medskip
(ii) On the other hand, each morphism $h$ of $\mathscr{A}\rtimes_{\sigma}\mathbf{Z}$
can be extended to
$E_0\cong \left(\mathscr{A}\rtimes_{\sigma}\mathbf{Z}\right)\oplus\mathbf{C}$
by the formula
\begin{equation}\label{eq3.3}
h(E_0)= h\left(\mathscr{A}\rtimes_{\sigma}\mathbf{Z}\right)\oplus\mathbf{C}. 
\end{equation}
 Therefore we get a functor 
from  the category of groups $G(\mathbb{A})$ and such of the $C^*$-algebras
$\left(\mathscr{A}\rtimes_{\sigma}\mathbf{Z}\right)\oplus\mathbf{C}$.
In other words, the  $E_0$ is a non-commutative coordinate
ring of the group $G(\mathbb{A})$.  
\end{proof}

\begin{lemma}\label{lm3}
There exists a one-to-one map 
\begin{equation}\label{eq3.4}
\phi:  G(\mathbb{A}_{\infty})\backslash G(\mathbb{A})/G(K)\to  K_0^{\mathbf{set}}(\mathscr{A}\rtimes_{\sigma}\mathbf{Z}),
\end{equation}
where  $K_0^{\mathbf{set}}(\mathscr{A}\rtimes_{\sigma}\mathbf{Z})$ is the set of elements of the 
group  $K_0(\mathscr{A}\rtimes_{\sigma}\mathbf{Z})$. 
\end{lemma}
\begin{proof}
Consider the exact sequence (\ref{eq3.1}) of the $C^*$-algebras:
\begin{equation}\label{eq3.5}
1\to\mathbf{C}\to E_0\to \mathscr{A}\rtimes_{\sigma}\mathbf{Z}\to 1,
\end{equation}
where $E_0\cong\left(\mathscr{A}\rtimes_{\sigma}\mathbf{Z}\right)\oplus\mathbf{C}$.
The exact sequence (\ref{eq3.5}) gives rise to the exact sequence of the  $K_0$-groups:
\begin{equation}\label{eq3.6}
1\to\mathbf{Z}\to K_0(E_0)\to K_0(\mathscr{A}\rtimes_{\sigma}\mathbf{Z})\to 1,
\end{equation}
where $K_0(\mathbf{C})\cong\mathbf{Z}$.

\medskip
In view of lemma \ref{lm2},  we have a functor $F$ acting by the formula $G(\mathbb{A})\mapsto  E_0$. 
On the other hand, the $K_0$ is a covariant functor from the category of $C^*$-algebras to the category of
 abelian groups. Thus, the composition $K_0\circ F$ defines a functor acting by the formula 
  $G(\mathbb{A})\mapsto K_0(E_0)$. 

\begin{figure}
\begin{picture}(300,110)(80,0)
\put(140,70){\vector(0,-1){35}}
\put(270,70){\vector(0,-1){35}}

\put(175,83){\vector(1,0){40}}
\put(175,23){\vector(1,0){40}}

\put(150,50){$K_0\circ F$}
\put(250,50){$\Phi$}
\put(125,20){$K_0(E_0)$}
\put(240,20){$K_0^{\mathbf{set}}(\mathscr{A}\rtimes_{\sigma}\mathbf{Z})$}

\put(125,80){$G(\mathbb{A})$}
\put(230,80){$G(\mathbb{A}_{\infty})\backslash G(\mathbb{A})/G(K)$}

\end{picture}
\caption{Double cosets  $G(\mathbb{A}_{\infty})\backslash G(\mathbb{A})/G(K)$ }
\end{figure}

\bigskip
Using the  exact sequence (\ref{eq3.6}),   one can define a map 
$K_0(E_0)\to K_0^{\mathbf{set}}(\mathscr{A}\rtimes_{\sigma}\mathbf{Z})$
from the abelian group to a finite set.  On the other hand, we  have 
a double coset map $G(\mathbb{A})\to G(\mathbb{A}_{\infty})\backslash G(\mathbb{A})/G(K)$
from an algebraic group to a finite set, see Section 2.2.   Bringing these facts together,  one gets a commutative 
diagram in Figure 1. 

\medskip
It remains to notice, that $\Phi$ is a functor from the category of finite sets to itself. 
This observation implies  that $\Phi$ is a trivial functor, i.e.  $|G(\mathbb{A}_{\infty})\backslash G(\mathbb{A})/G(K)|=  
|K_0^{\mathbf{set}}(\mathscr{A}\rtimes_{\sigma}\mathbf{Z})|$.  In particular, there exists  a one-to-one map   
\linebreak
$\phi:  G(\mathbb{A}_{\infty})\backslash G(\mathbb{A})/G(K)\to K_0^{\mathbf{set}}(\mathscr{A}\rtimes_{\sigma}\mathbf{Z})$
between the corresponding finite sets. Lemma \ref{lm3} follows. 
\end{proof}

\begin{corollary}\label{cor1}
The map $\phi^{-1}$ defines the structure of an abelian group
on the set $G(\mathbb{A}_{\infty})\backslash G(\mathbb{A})/G(K)$
 thus extending   $\phi$ to  an isomorphism of the  abelian groups.
\end{corollary}
\begin{proof}
Recall that 
\begin{equation}\label{eq3.7}
K_0(\mathscr{A}\rtimes_{\sigma}\mathbf{Z})\cong {\mathbf{Z}^n\over (I-A)\mathbf{Z}^n}\cong
{\mathbf{Z}\over p_1^{n_1}\mathbf{Z}}\oplus\dots\oplus{\mathbf{Z}\over p_k^{n_k}\mathbf{Z}},
\end{equation}
where $A$ is an $n\times n$ integer matrix defining the AF-algebra $\mathscr{A}$, the $p_i$ are prime
and $n_i$ are positive integer numbers. 
One can take a generator $x_i$ in each cyclic group $\mathbf{Z}/p_k^{n_k}\mathbf{Z}$ and let 
$\phi^{-1}(x_i)$ be a generator of the abelian group structure on the set 
$G(\mathbb{A}_{\infty})\backslash G(\mathbb{A})/G(K)$.  It is clear that the map $\phi$ defines 
an isomorphism between the two abelian groups. 
Corollary \ref{cor1} is proved.
\end{proof}

Theorem \ref{thm1.1} follows form lemma \ref{lm3} and corollary \ref{cor1}.

\subsection{Proof of corollary \ref{cor1.2}}
Recall that if $G\cong O(q)$ is the orthogonal group of a quadratic form $q(x)$,
then the equivalence classes of $q(x)$ correspond one-to-one to the double cosets 
$G(\mathbb{A}_{\infty})\backslash G(\mathbb{A})/G(K)$, see Section 2.2.4.  But in view 
of the corollary \ref{cor1}, the set $G(\mathbb{A}_{\infty})\backslash G(\mathbb{A})/G(K)$
has the natural structure of an abelian group defined by the formula (\ref{eq3.7}). 
In particular, the group operation defines a composition law for the equivalence 
classes of  quadratic  form $q(x)$.  The first part of corollary \ref{cor1.2} is proved. 

To express the genus $g$ of the form $q(x)$ in terms of an invariant of the algebra $\mathscr{A}$, 
recall that $g=|G(\mathbb{A}_{\infty})\backslash G(\mathbb{A})/G(K)|$,  see Section 2.2.4. 
But  $|G(\mathbb{A}_{\infty})\backslash G(\mathbb{A})/G(K)|=|K_0(\mathscr{A}\rtimes_{\sigma}\mathbf{Z})|$ 
and one gets from the formulas  (\ref{eq3.7}):  
\begin{eqnarray}\label{eq3.8}
g &=& |K_0(\mathscr{A}\rtimes_{\sigma}\mathbf{Z})| = \left|{\mathbf{Z}^n\over (I-A)\mathbf{Z}^n}\right| =
  \left|{\mathbf{Z}\over p_1^{n_1}\mathbf{Z}}\oplus\dots\oplus{\mathbf{Z}\over p_k^{n_k}\mathbf{Z}}\right|=\nonumber\\
  &=&  p_1^{n_1}\dots p_k^{n_k}=|\det ~(I-A)|. 
\end{eqnarray}
The genus formula of corollary \ref{cor1.2}  follows from the equations (\ref{eq3.8}).

\section{Binary quadratic forms}
Let $n=2$ and $\Delta>0$ be a square-free integer, i.e. $q_{\Delta}(x)$ is an indefinite binary quadratic form.
In view of corollary \ref{cor1.2},  we are looking for a matrix $A$ with $\det~A=1$, 
such that
\begin{equation}\label{eq4.1}
g(q_{f^2\Delta}(x))=|\det~(I-A^k)|,
\end{equation}
where $f\ge 1$ is a conductor of the quadratic form and $k\ge 1$ is an integer. 
\begin{remark}
The left and right sides of equation (\ref{eq4.1}) depend only on the integers $f$ and $k$,
respectively. Indeed, the left side depends only on the conductor $f$,  if the discriminant $\Delta$
is a fixed square-free integer. For the right side,  we have the easily verified equalities 
$|\det~(I-A^k)|=tr~(A^k)-2=2T_k \left({1\over 2}tr~(A)\right)-2=
2T_k\left({1\over 2}\sqrt{\Delta+4}\right)-2$, where $T_k(x)$ is the Chebyshev polynomial of degree $k$. 
\end{remark}
Let $f$ and $k$ be the least integers satisfying equation (\ref{eq4.1}).
Recall that the equivalence classes of the quadratic form $q_{f^2\Delta}(x)$
correspond one-to-one to the  (narrow) ideal classes of the order $R_f:=\mathbf{Z}+fO_K$
in the real quadratic field $K:=\mathbf{Q}(\sqrt{\Delta})$.  The number $h_{R_f}$
of such classes is calculated by the formula:
\begin{equation}\label{eq4.2}
h_{R_f}=h_K  {f \over e_f}\prod_{p|f}\left(1-\left({\Delta\over p}\right){1\over p}\right),
\end{equation}
 where $h_K$ is the class number of the field $K$,  $e_f$ is the index
of the group of units of  the order $R_f$ in the group of units of the ring $O_K$, $p$ is a prime number and 
$\left({\Delta\over p}\right)$ is the Legendre symbol  [Hasse 1950]  \cite[pp. 297 and 351]{H}. 
Thus the genus of the binary quadratic form of discriminant $\Delta>0$ is given by the formula:
\begin{equation}\label{eq4.3}
g(q_{\Delta}(x))={|\det~(I-A^k)| \over {f \over e_f}\prod_{p|f}\left(1-\left({\Delta\over p}\right){1\over p}\right)}.
\end{equation}

\begin{figure}
\begin{picture}(300,60)(0,0)
\put(108,27){$\bullet$}
\put(118,18){$\bullet$}
\put(138,18){$\bullet$}
\put(158,18){$\bullet$}
\put(118,38){$\bullet$}
\put(138,38){$\bullet$}
\put(158,38){$\bullet$}
\put(110,30){\line(1,1){10}}
\put(110,30){\line(1,-1){10}}
\put(140,42){\line(1,0){20}}
\put(140,39){\line(1,0){20}}
\put(140,20){\line(1,0){20}}
\put(120,20){\line(1,0){20}}
\put(120,42){\line(1,0){20}}
\put(120,39){\line(1,0){20}}
\put(120,40){\line(1,-1){20}}
\put(120,20){\line(1,1){20}}
\put(140,40){\line(1,-1){20}}
\put(140,20){\line(1,1){20}}

\put(180,20){$\dots$}
\put(180,40){$\dots$}

\end{picture}
\caption{Bratteli diagram of the AF-algebra $\mathscr{A}$.}
\end{figure}
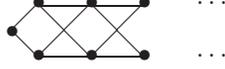

\begin{example}\label{ex4.2}
Let $G\cong O(q)$ be the orthogonal group of the quadratic form 
\begin{equation}\label{eq4.4}
q(u,v)=u^2+3uv+v^2.
\end{equation}
The discriminant  is $\Delta=5$ and therefore the $q(u,v)$ is an indefinite quadratic form.
The  AF-algebra $\mathscr{A}$ corresponding to the group $G(\mathbb{A})$  is represented by the matrix
\begin{equation}\label{eq4.5}
A=\left(
\begin{matrix}
2 & 1\cr
1 & 1
\end{matrix}
\right).
\end{equation}
The Bratteli diagram of algebra $\mathscr{A}$ is shown in Figure 2. 
To calculate the genus of $q(u,v)$, we shall use formula  (\ref{eq4.3}) with $f=k=1$.   
Namely, one gets
\begin{equation}\label{eq4.6}
g=|\det~(I-A)|=
\left|\det~
\left(
\begin{matrix}
-1 & -1\cr
-1 & 0
\end{matrix}
\right)
\right|=1.
\end{equation}
\end{example}

\bibliographystyle{amsplain}


\end{document}